\documentclass[12pt]{amsart}

\usepackage{amssymb,color}

\newtheorem{theorem}{Theorem}[section]

\newtheorem{lemma}[theorem]{Lemma}
\newtheorem{corollary}[theorem]{Corollary}
\newtheorem{conjecture}[theorem]{Conjecture}

\theoremstyle{definition}

\newtheorem{remark}[theorem]{Remark}

\newtheorem*{remark2}{Remark}

\newcommand{\ZZ}{ \ensuremath{\mathbb{Z}}}

\newcommand{\FF}{ \ensuremath{\mathbb{F}}}

\newcommand{\depth}{{\mathop{\mathrm{depth}}}}

\newcommand{\Ker}{\mathop{\mathrm{Ker}}}
\newcommand{\Image}{\mathop{\mathrm{Im}}}

\newcommand{\lk}{{\mathrm{lk}}}
\newcommand{\st}{\mathrm{st}}
\newcommand{\sd}{\mathrm{sd}}

\newcommand{\mideal}{\ensuremath{\mathfrak{m}}}
\newcommand{\field}{\ensuremath{\mathbb F}}

\def\cocoa{{\hbox{\rm C\kern-.13em o\kern-.07em C\kern-.13em o\kern-.15em A}}}

\newcommand{\ddd}{\delta} %{\mathbf{d}}

\begin{document}

\title[A Duality in Buchsbaum rings and triangulated manifolds]{A Duality in Buchsbaum rings and \\ triangulated manifolds}

\author{Satoshi Murai}\thanks{Murai's research is partially
supported by Grant-in-Aid for Scientific Research (C) 25400043}
\address{
Satoshi Murai,
Department of Pure and Applied Mathematics,
Graduate School of Information Science and Technology,
Osaka University,
Suita, Osaka, 565-0871, Japan}
\email{s-murai@ist.osaka-u.ac.jp
}

\author{Isabella Novik}\thanks{Novik's research is partially
supported by NSF grant DMS-1361423}
\address{
Isabella Novik,
Department of Mathematics,
University of Washington,
Seattle, WA 98195-4350, USA.
}
\email{novik@math.washington.edu}

\author{Ken-ichi Yoshida}\thanks{Yoshida's research is partially 
supported by Grant-in-Aid for Scientific Research (C) 25400050}
\address{Ken-ichi Yoshida, Department of Mathematics, College of Humanities and Sciences,
Nihon University, Setagaya-ku, Tokyo, 156-8550, Japan
}
\email{yoshida@math.chs.nihon-u.ac.jp}

%Keyword and Subject Classes (if needed)
%\keywords{}
%\subjclass[2000]{}
%\dedicatory{Dedicated to on the occasion of his birthday}

\begin{abstract}
Let $\Delta$ be a triangulated homology ball whose boundary complex is $\partial\Delta$. A result of Hochster asserts that the canonical module of the Stanley--Reisner ring of $\Delta$, $\FF[\Delta]$,  is isomorphic to the Stanley--Reisner module of the pair $(\Delta, \partial\Delta)$,  $\field[\Delta,\partial \Delta]$. This result implies that an  Artinian reduction of $\FF[\Delta,\partial \Delta]$ is (up to a shift in grading) isomorphic to the Matlis dual of the corresponding Artinian reduction of $\FF[\Delta]$. We establish a generalization of this duality to all triangulations of connected orientable homology manifolds with boundary.
We also provide an explicit algebraic interpretation of the $h''$-numbers of Buchsbaum complexes and use it to prove the monotonicity of $h''$-numbers for pairs of Buchsbaum complexes as well as the unimodality of $h''$-vectors of barycentric subdivisions of Buchsbaum polyhedral complexes. We close with applications to the algebraic manifold $g$-conjecture.
\end{abstract}
\maketitle
\section{Introduction}

In this paper, we study an algebraic duality of Stanley--Reisner rings of triangulated homology manifolds with non-empty boundary.
Our starting point is the following (unpublished) result of Hochster --- see \cite[Ch.~II, \S 7]{St}. (We defer most of definitions until later sections.) Let $\Delta$ be a triangulated $(d-1)$-dimensional homology ball whose boundary complex is $\partial\Delta$, let $\FF[\Delta]$ be the Stanley--Reisner ring of $\Delta$, and let $\FF[\Delta,\partial\Delta]$ be the Stanley--Reisner module of $(\Delta,\partial\Delta)$.
(Throughout the paper $\field$ denotes an infinite field.) 
Hochster's result asserts that the canonical module $\omega_{\FF[\Delta]}$ of $\field[\Delta]$ is isomorphic to $\FF[\Delta,\partial\Delta]$. In the last decade or so, this result had a lot of impact on the study of face numbers of simplicial complexes, especially in connection with the $g$-conjecture for spheres, see, for instance, the proof of Theorem 3.1 in a recent survey paper by Swartz \cite{Sw}. 

One numerical consequence of Hochster's result is the following symmetry of $h$-numbers of homology balls: $h_i(\Delta, \partial\Delta)=h_{d-i}(\Delta)$. The $h$-numbers are certain linear combinations of the face numbers; they are usually arranged in a vector called the $h$-vector. In fact, Hochster's result implies a  stronger statement:
it implies that there is an isomorphism
\begin{align}
\label{ball}
\field[\Delta,\partial \Delta]/ \Theta \field[\Delta,\partial \Delta]
\cong \big(\field[\Delta]/\Theta \field[\Delta]\big)^\vee (-d),
\end{align}
where $\Theta$ is a linear system of parameters for $\field[\Delta]$
and $N^\vee$ is the (graded) Matlis dual of $N$ (see e.g.\ \cite[Lemma 3.6]{MY}). As the $\FF$-dimensions of the $i$th graded components of modules in \eqref{ball} are equal to $h_i(\Delta,\partial \Delta)$ and $h_{d-i}(\Delta)$, respectively, the above-mentioned symmetry,  $h_i(\Delta, \partial\Delta)=h_{d-i}(\Delta)$, follows.

Hochster's result was generalized to homology manifolds with boundary by Gr\"abe \cite{Gr84}.
To state Gr\"abe's result, we recall the definition of homology manifolds.
We denote by $\mathbb{S}^d$ and $\mathbb B^d$ the $d$-dimensional sphere and ball, respectively. A pure $d$-dimensional simplicial complex $\Delta$ is an \textit{$\field$-homology $d$-manifold without boundary} if the link of each nonempty face $\tau$ of $\Delta$ has the homology of  $\mathbb{S}^{d-|\tau|}$ (over $\field$) .
An \textit{$\field$-homology $d$-manifold with boundary} is a pure $d$-dimensional simplicial complex $\Delta$ such that (i) the link of each nonempty face $\tau$ of $\Delta$ has the homology of either $\mathbb{S}^{d-|\tau|}$ or $\mathbb{B}^{d-|\tau|}$, and (ii) the set of all \textit{boundary faces}, that is,
\[
\partial\Delta:= \left\{\tau\in\Delta \ : \ 
\mbox{the link of $\tau$ has the same homology as $\mathbb{B}^{d-|\tau|}$}
\right\} \cup \{\emptyset\}
\]
is a $(d-1)$-dimensional $\field$-homology manifold without boundary. 
A connected  $\field$-homology $d$-manifold with boundary is said to be \textit{orientable} if the top homology  $\widetilde {H}_d(\Delta,\partial\Delta)$ is isomorphic to $\field$.

Gr\"abe \cite{Gr84} proved that if $\Delta$ is an orientable homology manifold with boundary, then $\field[\Delta,\partial \Delta]$ is the canonical module of $\field [\Delta]$.
Gr\"abe also established a symmetry of $h$-numbers for such a $\Delta$ (see \cite{Gr87}). While Gr\"abe's original statement of symmetry is somewhat complicated, it was recently observed by the first two authors \cite{MN} that it takes the following simple form when expressed in the language of $h''$-numbers.

\begin{theorem}
\label{MuraiNovik}
Let $\Delta$ be a connected orientable $\field$-homology $(d-1)$-manifold with non-empty boundary $\partial \Delta$. Then $h_i''(\Delta,\partial \Delta)=h_{d-i}''(\Delta)$ for all $i=0,1,\dots,d$.
\end{theorem}

The $h''$-numbers are certain modifications of $h$-numbers (see Section 3 for their definition). Similarly to the $h$-numbers, the $h''$-numbers are usually arranged in a vector, called the $h''$-vector. For homology manifolds, this vector appears to be a ``correct" analog of the $h$-vector. Indeed, many properties of $h$-vectors of homology balls and spheres are now known to hold for the $h''$-vectors of homology manifolds (with and without boundary), see recent survey articles \cite{KN,Sw}. 
In light of Gr\"abe's result from \cite{Gr84} and Theorem \ref{MuraiNovik},
it is natural to ask if Theorem \ref{MuraiNovik}
can be explained by Matlis duality.
The first goal of this paper is to provide such an explanation.

To this end, the key object is the submodule  $\Sigma(\Theta;M)$
defined by Goto \cite{Go}. Several definitions are in order.
Let $S=\field[x_1,\dots,x_n]$ be a graded polynomial ring over a field $\field$ with $\deg x_i=1$ for $i=1,2,\dots,n$.
Let $M$ be a finitely generated graded $S$-module of Krull dimension $d$ and
let $\Theta=\theta_1,\dots,\theta_d$ be a homogeneous system of parameters for $M$.
The module $\Sigma(\Theta;M)$ is defined as follows:
$$\Sigma(\Theta;M)=\Theta M + \left(\sum_{i=1}^d (\theta_1,\dots,\hat \theta_i,\dots,\theta_d) M :_M \theta_i\right) \subseteq M.$$
This module was introduced by Goto in \cite{Go} and has been used in the study of Buchsbaum local rings. Note that if $M$ is a Cohen--Macaulay module, then $\Sigma(\Theta;M)=\Theta M$.
We first show that this submodule is closely related to the $h''$-vectors.
Specifically, we establish the following explicit algebraic interpretation of $h''$-numbers.

\begin{theorem}
\label{3.3}
Let $(\Delta,\Gamma)$ be a Buchsbaum relative simplicial complex of dimension $d-1$ and let $\Theta$ be a linear system of parameters for $\field[\Delta,\Gamma]$. Then
$$\dim_\field \big( \field[\Delta,\Gamma]/ \Sigma \big(\Theta;\field[\Delta,\Gamma])\big)_j =h_j''(\Delta,\Gamma) \quad \mbox{ for all $j=0,1,\dots,d$}.$$
\end{theorem}

Theorems \ref{MuraiNovik} and \ref{3.3} suggest
that when $\Delta$ is a homology manifold with boundary there might be a duality between the quotients of
$\field[\Delta]$ and $\field[\Delta,\partial \Delta]$ by $\Sigma(\Theta;\field[\Delta])$ and
$\Sigma(\Theta; \field[\Delta,\partial \Delta])$, respectively.
We prove that this is indeed the case.
In fact, we prove a more general algebraic result on canonical modules of Buchsbaum graded algebras.
Let $\mideal=(x_1,\dots,x_n)$ be the graded maximal ideal of $S$.
If $M$ is a finitely generated graded $S$-module of Krull dimension $d$,
then the \textit{canonical module} of $M$, $\omega_M$, is the module
$$\omega_M := \left( H_\mideal^d(M)\right)^\vee,$$
where $H_\mideal^i(M)$ denotes the $i$th local cohomology module of $M$.
We prove that the following isomorphism holds for {\em all} Buchsbaum graded algebras.

\begin{theorem}\label{BBMdual}
Let $R=S/I$ be a Buchsbaum graded $\field$-algebra of Krull dimension $d \geq 2$, let
$\Theta=\theta_1,\dots,\theta_d \in S$ be a homogeneous system of parameters for $R$,
and let $\ddd=\sum_{i=1}^d \deg \theta_i$.
If $\depth\ \! R \geq 2$, then 
$$ \omega_R/\Sigma(\Theta;\omega_R) 
\cong \big(R/\Sigma(\Theta;R) \big)^\vee (-\ddd).$$
\end{theorem}

As we mentioned above,
if $\Delta$ is a connected orientable homology manifold, then (by Gr\"abe's result) the module $\field[\Delta,\partial \Delta]$ is the canonical module of $\field [\Delta]$; furthermore it is not hard to see that  $\field[\Delta]$ satisfies the assumptions of Theorem \ref{BBMdual}. (Indeed, the connectivity of $\Delta$ implies that $\depth\ \! \field[\Delta]\geq 2$, and the fact that $\field[\Delta]$ is Buchsbaum follows from Schenzel's theorem --- see Theorem \ref{3.1} below.) 
Hence we obtain the following corollary
that generalizes \eqref{ball}.

\begin{corollary}
\label{MFDdual}
Let $\Delta$ be a connected orientable $\field$-homology $(d-1)$-manifold with non-empty boundary $\partial \Delta$ and let $\Theta$ be a linear system of parameters for $\field[\Delta]$.
Then
$$ \field[\Delta,\partial \Delta]/\Sigma(\Theta; \field[\Delta,\partial \Delta] ) \cong  \big( \field [\Delta] / \Sigma(\Theta;\field[\Delta]) \big) ^\vee(-d).$$
\end{corollary}

We also consider combinatorial and algebraic applications of Theorems \ref{3.3} and \ref{BBMdual}. Specifically, we prove the monotonicity of $h''$-vectors for pairs of Buchsbaum simplicial complexes, establish the unimodality of $h''$-vectors of barycentric subdivisions of Buchsbaum polyhedral complexes, provide a combinatorial formula for the $a$-invariant of Buchsbaum Stanley--Reisner rings, and extend the result of Swartz \cite[Theorem 3.1]{Sw} as well as the result of B\"ohm and Papadakis \cite[Corollary 4.5]{BP} related to the sphere $g$-conjecture to the generality of the manifold $g$-conjecture. More precisely, Swartz's result asserts that most of bistellar flips when applied to a homology sphere preserve the weak Lefschetz property while B\"ohm--Papadakis' result asserts that stellar subdivisions at large-dimensional faces of homology spheres preserve the weak Lefschetz property; we extend both of these results to bistellar flips and stellar subdivisions performed on connected orientable homology manifolds.

The structure of the paper is as follows. In Section 2 we prove Theorem \ref{BBMdual} (although we defer part of a proof to the Appendix). In Section 3, we study Stanley--Reisner rings and modules of Buchsbaum simplicial
complexes. There, after reviewing basics of simplicial complexes and Stanley--Reisner
rings and modules, we verify Theorem \ref{3.3} and derive several combinatorial consequences.
Section 4 is devoted to applications of our results to the manifold $g$-conjecture. Finally, in the Appendix, we prove a graded version of Goto's result \cite[Proposition 3.6]{Go} --- a result on which our proof of Theorem \ref{BBMdual} is based.

\section{Duality in Buchsbaum rings}

In this section, we prove Theorem \ref{BBMdual}.
We start by recalling some definitions and results pertaining to Buchsbaum rings and modules.

Let $S=\field[x_1,\dots,x_n]$ be a graded polynomial ring  with $\deg x_i=1$ for $i=1,2,\dots,n$ and let $\mideal=(x_1,\dots,x_n)$ be the graded maximal ideal of $S$.
Given a graded $S$-module $N$, we denote by $N(a)$ the module $N$ with grading shifted by $a \in \mathbb Z$, that is, $N(a)_j=N_{a+j}$.
If $M$ is a finitely generated graded $S$-module of Krull dimension $d$, then
a \textit{homogeneous system of parameters} (or h.s.o.p.) for $M$
is a sequence $\Theta=\theta_1,\dots,\theta_d \in \mideal$ of homogeneous elements such that $\dim_\field M/\Theta M < \infty$.
A sequence $\theta_1,\dots,\theta_r \in \mideal$ of homogeneous elements is said to be a \textit{weak $M$-sequence} if
$$(\theta_1,\dots,\theta_{i-1})M:_M \theta_i = (\theta_1,\dots,\theta_{i-1})M:_M\mideal$$
for all $i=1,2,\dots,r$.
We say that $M$ is \textit{Buchsbaum} if every h.s.o.p.~for $M$ is a weak $M$-sequence.

Let $M$ be a finitely generated graded $S$-module and let $\Theta$ be its h.s.o.p.
The function $P_{\Theta,M}:\ZZ \to \ZZ$ defined by
$$P_{\Theta,M}(n):= \dim_\field \big(M/ (\Theta)^{n+1} M \big)$$
is called the \textit{Hilbert--Samuel function of $M$ w.r.t.\ the ideal $(\Theta)$}.
It is known that there is a polynomial in $n$ of degree $d$, denoted by $p_{\Theta,M}(n)$, such that $P_{\Theta,M}(n)=p_{\Theta,M}(n)$ for $n \gg 0$ (see \cite[Proposition 4.6.2]{BH}).
The leading coefficient of the polynomial $p_{\Theta,M}(n)$ multiplied by $d!$ is called the \textit{multiplicity of $M$ w.r.t.\ the ideal $(\Theta)$}, and is denoted by $e_\Theta(M)$. We will make use of the following known characterization of Buchsbaum property,
see \cite[Theorem I.1.12 and Proposition I.2.6]{SV}.

\begin{lemma}
\label{2.1}
A finitely generated graded $S$-module $M$ of Krull dimension $d$ is Buchsbaum if and only if, for every h.s.o.p.\ $\Theta$ of $M$, 
$$\dim_\field \big(M/\Theta M\big)-e_\Theta(M)=\sum_{i=0}^{d-1} { d-1\choose i} \dim_\field H_\mideal^i(M).$$
\end{lemma}

We also recall some known results on canonical modules of Buchsbaum modules.
For a finitely generated graded $S$-module $M$,
the \textit{depth} of $M$ is defined by $\depth(M):= \min\{i:H_\mideal^i(M)\ne 0\}$.

\begin{lemma}
\label{2.2}
Let $M$ be a finitely generated graded $S$-module of Krull dimension $d$. If $M$ is Buchsbaum, then the following properties hold:
\begin{itemize}
\item[(i)] $\omega_M$ is Buchsbaum.
\item[(ii)] $H_\mideal^i(\omega_M) \cong (H_\mideal^{d-i+1}(M))^\vee$ (as graded modules)
for all $i=2,3,\dots,d-1$. 
\item[(iii)]
If $\depth(M) \geq 2$, then $(H_\mideal^d(\omega_M))^\vee \cong M$ (as graded modules).
\item[(iv)] If $\Theta$ is an h.s.o.p.\ for $M$, then $\Theta$ is also an h.s.o.p.\ for $\omega_M$
and $e_\Theta(\omega_M)=e_\Theta(M)$.
\end{itemize}
\end{lemma}

\noindent See \cite[Theorem II.4.9]{SV} for (i),
\cite[Korollar 3.13]{Sc82} or \cite[(1.16)]{AG} for (ii) and (iii),
and \cite[Lemma 2.2]{Su} for (iv).

The following theorem is a graded version of Goto's
result \cite[Proposition 3.6]{Go}. The original result by Goto
is a statement about Buchsbaum {\em local rings}.
It may be possible to prove Theorem \ref{2.3} in the same way as in \cite{Go} by replacing rings with modules and by carefully keeping track of grading. 
However, since we could not find any literature allowing us to easily check this statement, we will provide its proof in the Appendix.

\begin{theorem}
\label{2.3}
Let $M$ be a finitely generated graded $S$-module of Krull dimension $d>0$. Assume further that $M$ is Buchsbaum and that
$\Theta=\theta_1,\dots,\theta_d$ is an h.s.o.p.\ for $M$ with $\deg \theta_i=\ddd_i$.
Let $\ddd_C := \sum_{i \in C} \ddd_i$ for $C \subseteq [d]=\{1,2,\dots,d\}$. Then
\begin{itemize}
\item[(i)] $\Sigma(\Theta;M)/\Theta M \cong \bigoplus_{C \subsetneq [d]} H_\mideal^{|C|} (M)(-\ddd_C)$, and
\item[(ii)] there is an injection $M/\Sigma(\Theta;M) \to H_\mideal^d(M)(-\ddd_{[d]})$.
\end{itemize}
\end{theorem}

We are now in a position to prove Theorem \ref{BBMdual}.

\begin{proof}[Proof of Theorem \ref{BBMdual}]
We first prove that $R/\Sigma(\Theta;R)$ and $\omega_R/\Sigma(\Theta;\omega_R)$ have the same $\field$-dimension.
Indeed, 
\begin{align*}
\dim_\field \big( R /\Sigma(\Theta;R)\big)
&= \dim_\field (R/\Theta R) -\dim_\field \big(\Sigma(\Theta;R)/\Theta R\big)\\
&= e_\Theta(R) + \sum_{i=0}^{d-1} {d-1 \choose i} \dim_\field H_\mideal^i (R)-\sum_{i=0}^{d-1} { d \choose i} \dim_\field H_\mideal^i (R)\\
&=e_\Theta(R) -\sum_{i=1}^{d-1} {d-1 \choose i-1} \dim_\field H_\mideal^i(R),
\end{align*}
where we use Lemma \ref{2.1} and Theorem \ref{2.3}(i) for the second equality.
Similarly, since $\omega_R$ is Buchsbaum, the same computation yields
$$
\dim_\field \big( \omega_R /\Sigma(\Theta;\omega_R)\big)
=e_\Theta(\omega_R) -\sum_{i=1}^{d-1} {d-1 \choose i-1} \dim_\field H_\mideal^i(\omega_R).$$
Now, since $\depth(R) \geq 2$ by the assumptions of the theorem
and since $\depth(\omega_R)\geq 2$ always holds, see \cite[Lemma 1]{Ao80},
parts (ii) and (iv) of Lemma \ref{2.2} guarantee that
$$\dim_\field \big(R/\Sigma(\Theta;R) \big)= \dim_\field \big(\omega_R/\Sigma(\Theta;\omega_R)\big).$$

Thus, to complete the proof of the statement, it suffices to show that there is a surjection
from $R/\Sigma(\Theta;R)(+\ddd)$ to $(\omega_R/\Sigma(\Theta;\omega_R))^\vee$.
By Theorem \ref{2.3}(ii), there is an injection
$$\omega_R/\Sigma(\Theta;\omega_R) \to H_\mideal^d (\omega_R) (-\ddd).$$
 Dualizing and using Lemma \ref{2.2}(iii), we obtain a surjection
\begin{align}
\label{2-1}
R(+\ddd) \cong \left( H_\mideal^d(\omega_R)(-\ddd) \right)^\vee \to \big(\omega_R/\Sigma(\Theta;\omega_R)\big)^\vee.
\end{align}
Since $\omega_R$ is an $R$-module, it follows from the definition of $\Sigma(\Theta;\omega_R)$ that 
$\Sigma(\Theta;R)\cdot \omega_R \subseteq \Sigma(\Theta;\omega_R)$. Hence $$\Sigma(\Theta;R) \cdot \big(\omega_R/\Sigma(\Theta;\omega_R)\big)=0,$$ which in turn implies
\begin{align}
\label{2-2}
\Sigma(\Theta;R) \cdot \big(\omega_R/\Sigma(\Theta;\omega_R)\big)^\vee=0.
\end{align}
Now \eqref{2-1} and \eqref{2-2} put together guarantee the existence of a surjection
$$\big(R/\Sigma(\Theta;R)\big)(+\ddd) \to \big(\omega_R/\Sigma(\Theta;\omega_R)\big)^\vee,$$
as desired.
\end{proof}

\begin{remark}
The above proof also works in the local setting: it
shows that if $R$ is a Buchsbaum Noetherian local ring with $\depth\ \! R\geq 2$ (and if the canonical module of $R$ exists),
then $\omega_R/\Sigma(\Theta;\omega_R)$ is isomorphic to the Matlis dual of $R/\Sigma(\Theta;R)$. 
\end{remark}

\section{Duality in Stanley--Reisner rings of manifolds}

In this section, we study Buchsbaum Stanley--Reisner rings and modules.
Some objects in this section  such as homology groups, Betti numbers,  $h'$- and $h''$-numbers depend on the characteristic of $\field$; however, we fix  a field $\field$ throughout this section, and omit $\field$ from our notation.

We start by reviewing basics of simplicial complexes and Stanley--Reisner rings. A \textit{simplicial complex} $\Delta$ on $[n]$ is a collection of subsets of $[n]$ that is closed under inclusion. A \textit{relative simplicial complex} $\Psi$ on $[n]$ is a collection of subsets of $[n]$ with the property that there are simplicial complexes $\Delta \supseteq \Gamma$ such that $\Psi=\Delta\setminus \Gamma$. We identify such a pair of simplicial complexes $(\Delta,\Gamma)$ with the relative simplicial complex $\Delta \setminus \Gamma$.
Also, a simplicial complex $\Delta$ will be identified with $(\Delta,\emptyset)$.
A \textit{face} of $(\Delta,\Gamma)$ is an element of $\Delta \setminus \Gamma$. The \textit{dimension} of a face $\tau$ is its cardinality minus one, and the dimension of $(\Delta,\Gamma)$ is the maximal dimension of its faces.
A relative simplicial complex is said to be \textit{pure} if all its maximal faces have the same dimension.

We denote by $\widetilde H_i(\Delta,\Gamma)$ the $i$th reduced homology group of the pair $(\Delta,\Gamma)$ computed with coefficients in $\field$: when $\Gamma\ne\emptyset$, $\widetilde H_*(\Delta,\Gamma)$ is the usual relative homology of a pair and when $\Gamma=\emptyset$, $\widetilde H_*(\Delta,\Gamma)=\widetilde H_*(\Delta)$ is the reduced homology of $\Delta$.
 The Betti numbers of $(\Delta,\Gamma)$ are defined by $\tilde \beta_i(\Delta,\Gamma):=\dim_\field \widetilde H_i(\Delta,\Gamma)$.
If $\Delta$ is a simplicial complex on $[n]$ and $\tau \in \Delta$ is a face of $\Delta$, then the \textit{link} of $\tau$ in $\Delta$ is 
$$\lk_\Delta(\tau):=\{\sigma \in \Delta: \tau\cup\sigma \in \Delta,\ \tau \cap \sigma = \emptyset\}.$$
For convenience, we also define $\lk_\Delta(\tau)=\emptyset$ if $\tau \not \in \Delta$.

Let $\Delta$ be a simplicial complex on $[n]$.
The Stanley--Reisner ideal of $\Delta$ (in $S$) is the ideal
$$I_\Delta=(x^\tau: \tau \subseteq [n],\ \tau \not \in \Delta) \subseteq S,$$
where $x^\tau = \prod_{i \in \tau} x_i$.
If $(\Delta,\Gamma)$ is a relative simplicial complex, then the \textit{Stanley--Reisner module} of $(\Delta,\Gamma)$ is the $S$-module
$$\field[\Delta,\Gamma]=I_\Gamma/I_\Delta.$$
When $\Gamma = \emptyset$, the ring $\field [\Delta]=\field[\Delta,\emptyset]=S/I_\Delta$ is called the \textit{Stanley--Reisner ring} of $\Delta$.

A relative simplicial complex $(\Delta,\Gamma)$ is said to be \textit{Buchsbaum} if $\field[\Delta,\Gamma]$ is a Buchsbaum module.
The following characterization of Buchsbaum property was given by Schenzel \cite[Theorem 3.2]{Sc}; a proof for relative simplicial complexes appears in \cite[Theorem 1.11]{AS}.

\begin{theorem}[Schenzel]  \label{3.1}
A pure relative simplicial complex $(\Delta,\Gamma)$ of dimension $d$ is Buchsbaum if and only if, $\widetilde H_i\big(\lk_\Delta(\tau),\lk_\Gamma(\tau)\big)=0$ for every non-empty face $\tau \in \Delta\setminus \Gamma$ and all $i \ne d-|\tau|$.
\end{theorem}

\noindent In particular, if $\Delta$ is a homology manifold with boundary, then 
$\Delta$ and $(\Delta,\partial\Delta)$ are Buchsbaum (relative) simplicial complexes.  (However, most of Buchsbaum complexes are not homology manifolds.)

Next, we discuss face numbers of Buchsbaum simplicial complexes.
For a relative simplicial complex $(\Delta,\Gamma)$ of dimension $d-1$, let $f_i(\Delta,\Gamma)$ be the number of $i$-dimensional faces of $(\Delta,\Gamma)$ and let
$$h_j(\Delta,\Gamma)= \sum_{i=0}^j (-1)^{j-i} {d-i \choose d-j} f_{i-1}(\Delta,\Gamma) \quad \mbox{for $j=0,1,\dots,d$}. $$
For convenience, we also  define $h_j(\Delta,\Gamma)=0$ for $j> \dim(\Delta,\Gamma)+1$.
The $h$-numbers play a central role in the study of face numbers of Cohen--Macaulay simplicial complexes. On the other hand, for Buchsbaum simplicial complexes, the following modifications of  $h$-numbers, called $h'$-numbers and $h''$-numbers, behave better than the usual $h$-numbers.

Recall that a \textit{linear system of parameters} (or l.s.o.p.) is an h.s.o.p.\ $\Theta=\theta_1,\dots,\theta_d$ consisting of linear forms.
Note that when $\field$ is infinite, any finitely generated graded $S$-module has an l.s.o.p.
The following result, established by Schenzel \cite[\S 4]{Sc} (a proof for Stanley--Reisner modules appears in \cite[Theorem 2.5]{AS}), is known as Schenzel's formula.

\begin{theorem}[Schenzel]
\label{3.2}
Let $(\Delta,\Gamma)$ be a Buchsbaum relative simplicial complex of dimension $d-1$ and let $\Theta$ be an l.s.o.p.\ for $\field[\Delta,\Gamma]$. Then, for $j=0,1,\dots,d$,
$$\dim_\field \big(\field[\Delta,\Gamma]/ \Theta \field[\Delta,\Gamma] \big)_j=h_j(\Delta,\Gamma)
- {d \choose j} \sum_{i=1}^{j-1} (-1)^{j-i} \tilde \beta_{i-1}(\Delta,\Gamma).$$
\end{theorem}

In view of Schenzel's formula, we define the \textit{$h'$-numbers} of a $(d-1)$-dimensional relative simplicial complex $(\Delta,\Gamma)$ by
$$h_j'(\Delta,\Gamma)= h_j(\Delta,\Gamma) - {d \choose j} \sum_{i=1}^{j-1} (-1)^{j-i} \tilde \beta_{i-1}(\Delta,\Gamma).$$
Furthermore, we define the \textit{$h''$-numbers} of $(\Delta,\Gamma)$ by
\begin{align*}
h_j''(\Delta,\Gamma)=
\begin{cases}
h_j'(\Delta,\Gamma) - {d \choose j} \tilde \beta_{j-1}(\Delta,\Gamma), &\mbox{if } 0 \leq j <d,\\
h_d'(\Delta,\Gamma), &\mbox{if } j=d.
\end{cases}
\end{align*}

We are now ready to prove Theorem \ref{3.3}.

\begin{proof}[Proof of Theorem \ref{3.3}]
Let $M=\field[\Delta,\Gamma]$.
Observe that
\begin{align*}
\dim_\field \big(M/\Sigma(\Theta;M)\big)_j
&= \dim_\field (M/\Theta M)_j - \dim_\field \big(\Sigma(\Theta;M)/\Theta M\big)_j.
\end{align*}
Since $H_\mideal^i(M)=(H_\mideal^i(M))_0\cong \widetilde H_{i-1}(\Delta,\Gamma)$ for $i<d$ (see \cite[Theorem 1.8]{AS}),
Theorem \ref{2.3}(i) implies
$$\dim_\field \big(\Sigma(\Theta;M)/\Theta M\big)_j= 
\left\{ \begin{array}{ll} 
{d \choose j} \tilde \beta_{j-1} (\Delta,\Gamma), & \mbox{ if $0 \leq j\leq d-1$},\\
0, & \mbox{ if $j\geq d$.}
\end{array}
\right.
$$
The desired statement then follows from Theorem \ref{3.2} asserting that $\dim_\field (M/\Theta M)_j= h_j'(\Delta,\Gamma)$ for all $j$.
\end{proof}

It was proved in \cite[Theorem 3.4]{NS} that 
the $j$th graded component of the socle of $\field[\Delta]/\Theta \field[\Delta]$ has dimension at least ${d \choose j} \tilde \beta_{j-1}(\Delta)$.
This implies that the $h''$-numbers of a Buchsbaum simplicial complex form the Hilbert function of some quotient of its Stanley--Reisner ring.
A new contribution and the significance of Theorem \ref{3.3} is that it provides an explicit algebraic interpretation of  $h''$-numbers via a submodule $\Sigma(\Theta;-)$.

In the rest of this section we discuss a few algebraic and combinatorial applications of our results.
As was proved by Kalai and, independently, Stanley (see \cite[Ch.~III, \S 9]{St} and \cite{St93}), if $\Delta \supseteq \Gamma$ are Cohen--Macaulay simplicial complexes of the same dimension, then $h_i(\Delta) \geq h_i(\Gamma)$ for all $i$.
The interpretation of the $h''$-numbers given in Theorem \ref{3.3} allows us to prove the following generalization of this fact.

\begin{theorem}
\label{3.7}
Let $\Delta \supseteq \Gamma$ be Buchsbaum simplicial complexes of the same dimension.
Then $h_i''(\Delta) \geq h_i''(\Gamma)$ for all $i$.
\end{theorem}

\begin{proof}
We may assume that $\Delta$ and $\Gamma$ are simplicial complexes on $[n]$
and that $\field$ is infinite.
Let $d=\dim \Delta +1$.
Then there is a common linear system of parameters $\Theta=\theta_1,\dots,\theta_d$ for $\field[\Delta]$ and $\field [\Gamma]$. By the definition of $\Sigma(\Theta;-)$, 
$$\textstyle
\field[\Delta]/\Sigma(\Theta;\field[\Delta]) = S/ \left((\Theta)+ \big(\sum_{k=1}^d \big( (\theta_1,\dots,\hat \theta_k,\dots,\theta_d)+ I_\Delta\big):_{S}\theta_k\big) \right)$$
and an analogous formula holds for $\field[\Gamma]/\Sigma(\Theta;\field[\Gamma]) $.
Since $I_\Gamma \supseteq I_\Delta$,
the above formula implies that $\field[\Gamma]/\Sigma(\Theta;\field[\Gamma])$
is a quotient ring of $\field[\Delta]/\Sigma(\Theta;\field[\Delta]) $.
The desired statement then follows from Theorem \ref{3.3}.
\end{proof}

\begin{corollary}
\label{4.2}
Let $\Delta$ be a Buchsbaum simplicial complex and let $\tau \in \Delta$ be any non-empty face. Then
$h''_i(\Delta) \geq h_i(\lk_\Delta(\tau))$ for all $i$.
\end{corollary}

\begin{proof}
Consider the star of $\tau$, $\st_\Delta(\tau)=\{ \sigma \in \Delta: \sigma \cup \tau  \in \Delta\}$.
Then  $\st_\Delta(\tau)$ is Cohen--Macaulay (by Theorem \ref{3.1} and Reisner's criterion) and has the same dimension as $\Delta$, and so by Theorem \ref{3.7}, $h''_i(\Delta)\geq h''_i(\st_\Delta(\tau))$ for all $i$. The result follows since
$\lk_\Delta(\tau)$ and $\st_\Delta(\tau)$ have the same $h$-numbers (this is because $\st_\Delta(\tau)$ is just a cone over $\lk_\Delta(\tau)$, cf.~\cite[Corollary III.9.2]{St}) and  since for Cohen--Macaulay simplicial complexes the $h''$-numbers coincide with the  $h$-numbers.
\end{proof}

Let $R=S/I$ be a graded $\field$-algebra. The \textit{$a$-invariant} of $R$ is the number
$$a(R)=-\min\{ k: (\omega_R)_k \ne 0\}.$$
This number is an important invariant in commutative algebra.
When $R$ is Cohen--Macaulay, it is well-known that 
$a(R)=\max\{i:h_i(R)\ne 0\}-\dim R$, where $h_i(R)$ is the $i$th $h$-number of $R$.
(See \cite[\S 4.1]{BH} for the definition of $h$-numbers for modules.)
The following result provides a generalization of this fact.

\begin{theorem}
\label{ainvariant}
Let $R=S/I$ be a Buchsbaum graded $\field$-algebra of Krull dimension $d$ with $\depth(R) \geq 2$ and let $\Theta=\theta_1,\dots,\theta_d$ be an l.s.o.p.~for $R$.
Then
$$
\textstyle a(R)= \max\{k: (R/\Sigma(\Theta;R))_k \ne 0\} - d .$$
In particular, for any connected Buchsbaum simplicial complex $\Delta$ of dimension $d-1$,
$a(\field[\Delta])=\max\{k: h''_k(\Delta) \ne 0\} -d$.
\end{theorem}

\begin{proof}
Let $m=\max\{ k: (R/\Sigma(\Theta;R))_k \ne 0\}$.
Then by Theorem \ref{BBMdual},
$$\min\{k: ( \omega_R/\Sigma(\Theta;\omega_R) )_k \ne 0\} =d - m.$$
As $\min\{k : (\omega_R) _k \ne 0\}=\min\{k : (\omega_R/\Theta \omega_R) _k \ne 0\}$, the theorem would follow if we prove that 
$\min\{k : (\omega_R/\Theta \omega_R) _k \ne 0\} = d -m.$
To this end, it is enough to show that
\begin{align}
\label{ainv1}
\big( \Sigma(\Theta;\omega_R)/\Theta \omega_R \big)_k =0 \ \ \ \mbox{ for } k \leq d - m-1.
\end{align}

Since $\mideal H_\mideal^i(R)=0$ for $i <d$ (see \cite[Proposition I.2.1]{SV}), we conclude from Theorem~\ref{2.3}(i) that
$\mideal (\Sigma(\Theta;R)/\Theta R)=0$.
Furthermore, since $(R/\Theta R)_k = (\Sigma(\Theta;R)/\Theta R)_k$ for $k \geq m+1$, it follows that for $k \geq m+2$,
$$
\big(\Sigma(\Theta;R)/\Theta R\big)_{k}
= (R/\Theta R)_{k}
=\big(\mideal(R/\Theta R)\big)_{k}
=\big(\mideal \big(\Sigma(\Theta;R)/\Theta R \big)\big)_{k}=0.$$
The isomorphism 
$$\Sigma(\Theta;R)/\Theta R \cong \bigoplus_{C \subsetneq [d]} H_\mideal^{|C|} (R) (-|C|)$$
established in Theorem \ref{2.3}(i) then implies that %for all $i \leq d-1$,
$$H_\mideal^{i}(R)_j=0 \quad \mbox{for all $i \leq d-1$ and $j\geq m+2-i$}.$$
Therefore, 
$$\big( H_\mideal^{d-i+1}(R)^\vee (-i) \big)_k=0 \quad
\mbox{for all $2 \leq i \leq d-1$ and $-(k- i)\geq m+2-(d-i+1) $},$$
and so 
\begin{align}
\label{ainv2}
\big( H_\mideal^{d-i+1}(R)^\vee (-i) \big)_k=0 \quad
\mbox{for all $2 \leq i \leq d-1$ and $k \leq  d - m -1$}.
\end{align}

Finally, since $\depth(\omega_R) \geq 2$, we infer from Theorem \ref{2.3}(i) and Lemma \ref{2.2}(ii) that
$$ \Sigma(\Theta;\omega_R)/\Theta \omega_R \cong \bigoplus_{C \subsetneq [d],\ \! |C| \geq 2}H_\mideal^{|C|} (\omega_R) (-|C|) \cong \bigoplus_{C \subsetneq [d],\ \! |C| \geq 2} H_\mideal^{d-|C|+1}(R)^\vee (-|C|).$$
The above isomorphisms and \eqref{ainv2} then yield the
desired property \eqref{ainv1}.
\end{proof}

Observe that by Hochster's formula on local cohomology \cite[Theorem 5.3.8]{BH}, if $\Delta$ is a $(d-1)$-dimensional Buchsbaum simplicial complex, then $-a(\FF[\Delta])$ equals the minimum cardinality of a face whose link has a non-vanishing top homology. Thus the ``in particular" part of Theorem \ref{ainvariant} can be equivalently restated as follows: if for every face $\tau\in\Delta$ of dimension $< d-k$, the link of $\tau$ has vanishing top homology, then $h''_{k}(\Delta)=0$. For the case of homology manifolds with boundary, faces with vanishing top homology are precisely the boundary faces; in this case, the above statement reduces to \cite[Theorem 3.1]{MNe}.

A sequence $h_0,h_1,\dots,h_m$ of numbers is said to be \textit{unimodal} if there is an index $p$ such that $h_0 \leq h_1\leq \cdots \leq h_p \geq \cdots \geq h_m$. It was proved in \cite{Mu} that the $h''$-numbers of the barycentric subdivision of any connected Buchsbaum simplicial complex form a unimodal sequence. Here we use the Matlis duality established in Theorem  \ref{BBMdual} to generalize this result to Buchsbaum polyhedral complexes.
(We refer our readers to \cite{Mu} for the definition of barycentric subdivisions.)

\begin{theorem}  \label{baryc-subdiv}
Let $\Delta$ be the barycentric subdivision of a connected polyhedral complex $\Gamma$ of dimension $d-1$.
Suppose that the characteristic of $\field$ is zero
and $\Delta$ is Buchsbaum. Then, for a generic choice of linear forms $\Theta=\theta_1,\dots,\theta_d$, and  $\theta_{d+1} \in \field[\Delta]$, the multiplication
$$\times \theta_{d+1}: \big(\field[\Delta]/\Sigma(\Theta;\field[\Delta])\big)_i
\to \big(\field[\Delta]/\Sigma(\Theta;\field[\Delta])\big)_{i+1}$$
is injective for $i \leq {\frac{d}{2}-1}$
and is surjective for $i \geq {\frac {d} 2}$.
In particular, the sequence $h_0''(\Delta),h_1''(\Delta),\dots,h_d''(\Delta)$ is unimodal.
\end{theorem}

\begin{proof}
By genericity of linear forms, $\Theta=\theta_1,\dots,\theta_d$
is a common l.s.o.p.\ for $\field[\Delta]$ and $\omega_{\field[\Delta]}$.
Let $P$ be the face poset of $\Gamma$.
Then the Stanley--Reisner ring $\FF[\Delta]$ is a squarefree $P$-module (a notion introduced in \cite[Definition 2.1]{MY}); furthermore, by \cite[Theorem 3.1]{MY} $\omega_{\field[\Delta]}$ is also a squarefree $P$-module.
It then follows from \cite[Theorem 6.2 (ii)]{MY} that the multiplication maps
\begin{equation} \label{surj-mult}
\times \theta_{d+1}: \big(\field[\Delta]/\Theta\field[\Delta]\big)_i
\to \big(\field[\Delta]/\Theta\field[\Delta]\big)_{i+1} 
\end{equation}
and
$$\times \theta_{d+1}: \big(\omega_{\field[\Delta]}/(\Theta\omega_{\field[\Delta]})\big)_i
\to \big(\omega_{\field[\Delta]}/(\Theta \omega_{\Theta\field[\Delta]})\big)_{i+1}.$$
are surjective for {$i \geq \frac {d} 2$}.
(While Cohen--Macaulayness was assumed in \cite[Theorem 6.2]{MY}, 
one can see from the proof given in \cite{MY} that this assumption was used only in the proof of part (i) and is unnecessary to derive surjectivity.)
Since $\Theta\field[\Delta]$ is contained in $\Sigma(\Theta;\field[\Delta])$,
%{\color{blue} as follows from Theorem \ref{2.3}(i) and \cite[Proposition I.2.1]{SV}, $\mideal \big(\Sigma(\Theta;\field[\Delta])/\Theta\field[\Delta]\big)=0$},
 % is contained in the socle of $\field[\Delta]/\Theta\field[\Delta]$,
the map in \eqref{surj-mult} remains surjective if we replace $\Theta\field[\Delta]$ with $\Sigma(\Theta;\field[\Delta])$,
and a similar statement holds for $\omega_{\field[\Delta]}$.
These surjectivities and the Matlis duality 
$\field[\Delta]/\Sigma(\Theta;\field[\Delta]) \cong
(\omega_{\field[\Delta]}/\Sigma(\Theta;\omega_{\field[\Delta]}))^\vee ({d})$ of Theorem \ref{BBMdual} yield the desired statement.
\end{proof}

In fact, in view of results from \cite{KM}, it is tempting to conjecture that if $\Delta$ is a barycentric subdivision of a Buchsbaum regular CW-complex of dimension $d-1$, then even the sequence $h''_0(\Delta)/{d \choose 0}, h''_1(\Delta)/{d \choose 1}, h''_2(\Delta)/{d \choose 2}, \ldots, h''_d(\Delta)/{d \choose d}$ is unimodal.

We close this section with a couple of remarks.

\begin{remark}
Our results on $h''$-vectors can be generalized to the following setting.
Consider a finitely generated graded $S$-module $M$ of Krull dimension $d$
such that %satisfying the condition
\begin{align}
\label{condition}
H^i_\mideal(M)=\big(H^i_\mideal(M)\big)_0\ \  \mbox{ for all $0 \leq i < d$}.
\end{align}
Define the $h'$- and $h''$-numbers of $M$ in the same way as for the Stanley--Reisner modules but with $\dim_\field (H_\mideal^{j}(M))$ used as a replacement for $\tilde \beta_{j-1}(\Delta,\Gamma)$. Then suitably modified statements of Theorems \ref{3.3} and \ref{3.2} continue to hold for such an $M$. 
Furthermore, if $M$ satisfies \eqref{condition}, then $M$ must be Buchsbaum (see \cite[Proposition I.3.10]{SV}), in which case $\omega_M$ also satisfies \eqref{condition} by Lemma \ref{2.2}(ii).
%Thus Corollary \ref{MFDdual} can also be extended to these modules.
In particular, if $\Delta$ is an arbitrary connected Buchsbaum simplicial complex of dimension $d-1$, then
$$h_i''(\omega_{\field[\Delta]})=h_{d-i}''(\field[\Delta]) \quad \mbox{for all $i=0,1,\ldots,d$}.$$
\end{remark}

\begin{remark}  \label{without-bdry}
A statement analogous to Corollary \ref{MFDdual} also holds for homology manifolds without boundary. Indeed, if $\Delta$ is an orientable homology manifold without boundary, then by Gr\"abe's result \cite{Gr84}, 
$\field[\Delta]$ is isomorphic to its own canonical module, and hence, by Theorem \ref{BBMdual}, $\field[\Delta]/\Sigma(\Theta;\field[\Delta])$ is an Artinian Gorenstein algebra.
This fact was essentially proved in \cite[Theorem 1.4]{NS2}.
%More generally,
%the result of Nagel \cite{Na} shows
%that a Buchsbaum simplicial complex $\Delta$ of dimension $d-1$ is Buchsbaum* (see \cite{AW}) if and only if 
%$\field[\Delta]/\Sigma(\Theta;\field[\Delta])$ is a level algebra of socle degree $d$,
%which generalizes a characterization of the doubly Cohen--Macaulay property (\cite[Ch.\ III, \S 3]{St}) in terms of the level property.

 Let us also point out that if  $\Delta$ is a connected orientable homology manifold with or without boundary, then for $M=\FF[\Delta]$, the statement of part (ii) of Lemma \ref{2.2}  is a simple consequence of the Poincar\'e--Alexander--Lefschetz duality along with Gr\"abe's result \cite{Gr84} that $\omega_M\cong \FF[\Delta,\partial\Delta]$. 
\end{remark}

\section{Applications to the manifold $g$-conjecture}

In this section we discuss connected orientable homology manifolds without boundary. One of the most important open problems in algebraic combinatorics is the algebraic $g$-conjecture; it asserts that every homology sphere has the weak Lefschetz property (the WLP, for short).  Kalai proposed a far-reaching generalization of this conjecture \cite[Conjecture 7.5]{No} to homology manifolds. Using the $\Sigma(\Theta,-)$ module allows to restate Kalai's conjecture as follows.

\begin{conjecture}
\label{mfdgcong}
Let $\Delta$ be a connected orientable $\FF$-homology $(d-1)$-manifold without boundary.
Then, for a generic choice of linear forms $\Theta=\theta_1,\dots,\theta_d$,
the ring $\FF[\Delta]/\Sigma(\Theta; \FF[\Delta])$ has the WLP, that is, for a generic linear form $\omega$, the multiplication %the multiplication by $\omega$,
\[
\times \omega \, : \, \big(\FF[\Delta]/\Sigma(\Theta; \FF[\Delta])\big)_{\lfloor d/2 \rfloor} \to \big(\FF[\Delta]/\Sigma(\Theta; \FF[\Delta])\big)_{\lfloor d/2 \rfloor + 1}
\]
is surjective.
\end{conjecture}

It is worth mentioning that while Kalai's original statement of the conjecture did not involve $\Sigma(\Theta,-)$, the two statements are equivalent  (see \cite{NS2} and Remark \ref{without-bdry} above). Somewhat informally, we say that $\Delta$ (or $\FF[\Delta]$) has the WLP if $\Delta$ satisfies the conclusions of the above conjecture.
Enumerative consequences of this conjecture are discussed in \cite[\S1]{NS2}. 

Given a finite set $A$, we denote by $\overline{A}$ the simplex on $A$, i.e., the simplicial complex whose set of faces consists of all subsets of $A$.  When $A=\{a\}$ consists of a single vertex, we write $\overline{a}$ to denote the vertex $a$, viewed as a $0$-dimensional simplex.  If $\Delta$ and $\Gamma$ are two simplicial complexes on disjoint vertex sets, then the {\em join} of $\Delta$ and $\Gamma$, $\Delta \ast \Gamma$, is the simplicial complex defined by
\[\Delta \ast \Gamma := \{ \sigma\cup \tau \,:\, \sigma\in \Delta, \tau\in\Gamma\}.\]
Finally, if $\Delta$ is a simplicial complex and $W$ is a set, then $\Delta_W=\{\sigma\in \Delta \, : \, \sigma\subseteq W\}$ is the subcomplex of $\Delta$ induced by $W$.

Let $\Delta$ be a $(d-1)$-dimensional homology manifold, and let $A$ and $B$ be disjoint 
subsets such that $|A|+|B|=d+1$. If $\Delta_{A\cup B}=\overline{A}\ast \partial\overline{B}$, then the operation of removing $\overline{A}\ast \partial\overline{B}$ from $\Delta$ and replacing it with $\partial\overline{A}\ast\overline{B}$ is called a {\em $(|B|-1)$-bistellar flip}. (For instance, a $0$-flip is simply a stellar subdivision at a facet.) The resulting complex is a homology manifold homeomorphic to the original complex.%$\|\Delta\|$. 

The following surprising property was proved by Pachner: if $\Delta_1$ and $\Delta_2$ are two PL homeomorphic combinatorial manifolds without boundary, then they can be connected by a sequence of bistellar flips (see \cite{Pachner-87,Pachner-91}). Since the boundary complex of a simplex has the WLP, Pachner's result suggests the following inductive approach to the algebraic $g$-conjecture for PL spheres: prove that bistellar flips applied to PL spheres preserve the WLP. 
In \cite[\S 3]{Sw}, Swartz showed that most of bistellar flips applied to homology {\em spheres} preserve the WLP. Here we extend Swartz's results to the generality of orientable homology manifolds without boundary:

\begin{theorem} \label{flips}
Let $\Delta$ be a $(d-1)$-dimensional, connected, orientable homology manifold without boundary. Suppose that $\Delta'$ is obtained from $\Delta$ via a $(p-1)$-bistellar flip with $p\neq (d+1)/2$ if $d$ is odd and with $p\notin\{d/2, (d+2)/2\}$ if $d$ is even, and let $\Theta$ be a common l.s.o.p.\ for $\FF[\Delta]$ and $\FF[\Delta']$. Then $\FF[\Delta']/\Sigma(\Theta; \FF[\Delta'])$ has the WLP if and only if $\FF[\Delta]/\Sigma(\Theta; \FF[\Delta])$ has the WLP.
\end{theorem}

If $\Delta$ is a simplicial complex and $\sigma$ is a face of $\Delta$, then the \textit{stellar subdivision} of $\Delta$ at $\sigma$ consists of (i) removing $\sigma$ and all faces containing it from $\Delta$, (ii) introducing a new vertex $a$, and (iii) adding new faces in $\overline{a}\ast\partial\overline{\sigma}\ast\lk_\Delta(\sigma)$ to $\Delta$:
\[
\sd_\sigma(\Delta):= \big(\Delta \setminus \st_\Delta(\sigma)\big) \cup \big(\overline{a}\ast\partial\overline{\sigma}\ast\lk_\Delta(\sigma)\big).
\] 
A classical result due to Alexander \cite{Alexander} asserts that two simplicial complexes are PL homeomorphic if and only if they are stellar equivalent, that is, one of them can be obtained from another by a sequence of stellar subdivisions and their inverses. Thus, a different approach to the algebraic $g$-conjecture (at least for PL manifolds) is to show that the WLP is preserved by stellar subdivisions and their inverses.
B\"ohm and Papadakis \cite[Corollary 4.5]{BP} proved that this is the case if one applies stellar subdivisions at faces of sufficiently large dimension (or their inverses) to homology spheres. We extend their result to the generality of orientable homology manifolds without boundary:

\begin{theorem} \label{stellar-subdiv}
Let $\Delta$ be a $(d-1)$-dimensional, connected, orientable $\FF$-homology manifold without boundary and let $\sigma$ be a face of $\Delta$ with $\dim\sigma>d/2$.
Then $\FF[\Delta]$ has the WLP if and only if $\FF[\sd_\sigma(\Delta)]$ has the WLP.
%, and let  $\Theta$ be a common l.s.o.p.\ for $\FF[\Delta]$ and $\FF[\sd_\sigma(\Delta)]$. Then $\FF[\Delta]/\Sigma(\Theta; \FF[\Delta])$ has the WLP if and only if $\FF[\sd_\sigma(\Delta)]/\Sigma(\Theta; \FF[\sd_\sigma(\Delta)])$ has the WLP.
\end{theorem} 

The structure of the proofs of both theorems is similar to the proof of \cite[Theorem 3.1]{Sw}. The new key ingredient is given by the following lemma.
Recall that $\Delta$ is an {\em $\FF$-homology $(d-1)$-sphere} if $\Delta$ is an $\FF$-homology manifold whose homology over $\FF$ coincides with that of $\mathbb{S}^{d-1}$. Similarly, $\Delta$ is an {\em $\FF$-homology ball} if (i) $\Delta$ is an $\FF$-homology manifold with boundary, (ii) the homology of $\Delta$ (over $\FF$) vanishes, and (iii) the boundary of $\Delta$ is an $\FF$-homology sphere.

\begin{lemma} \label{exact} Let $\Delta$ be a  $(d-1)$-dimensional $\FF$-homology manifold without boundary, $\Gamma$ a full-dimensional subcomplex of $\Delta$, and $\Theta$  an l.s.o.p.\ for $\FF[\Delta]$. Assume further that $\Gamma$ is an $\FF$-homology ball, and let $D$ be the simplicial complex obtained from $\Delta$ by removing the interior faces of $\Gamma$. Then $D$ is an $\FF$-homology manifold with boundary and the natural surjection $\FF[\Delta] \to \FF[\Gamma]$ induces the following short exact sequence
\[
0 \to \FF[D,\partial D]/\Sigma(\Theta;\FF[D,\partial D]) \to \FF[\Delta]/
\Sigma(\Theta; \FF[\Delta]) \to \FF[\Gamma]/\big(\Theta \FF[\Gamma]\big) \to 0.
\]
\end{lemma}

\begin{proof}
The fact that $D$ is a homology manifold with boundary follows by a standard Mayer-Vietoris argument (note that $\partial D=\partial\Gamma$ is a homology $(d-2)$-sphere).
Furthermore, by excision and since $\Gamma$ has vanishing homology,
\begin{equation} \label{Betti}
\tilde{\beta}_i(D, \partial D)=\tilde{\beta}_i(\Delta, \Gamma)=\tilde{\beta}_i(\Delta) \quad \mbox{for all $i$}.
\end{equation}

Now, using that the homology ball $\Gamma$ is a full-dimensional subcomplex of $\Delta$, we conclude as in the proof of Theorem \ref{3.7} that there is a natural surjection 
\begin{equation} \label{surjection}
\phi: \FF[\Delta]/\Sigma\big(\Theta;\FF[\Delta]\big) \to \FF[\Gamma]/ \big(\Theta \FF[\Gamma]\big)= \FF[\Gamma]/\Sigma\big(\Theta;\FF[\Gamma]\big).
\end{equation}
 Thus, to finish the proof of the lemma, it suffices to show that the kernel of $\phi$ is $\FF[D,\partial D]/\Sigma\big(\Theta; \FF[D,\partial D]\big)$. This, in turn, would follow if we verify the following two facts:
\begin{enumerate}
\item[(i)]  $\Ker(\phi)$ is a quotient of $\FF[D,\partial D]/\Sigma\big(\Theta; \FF[D,\partial D]\big)$, and
\item[(ii)]  $\dim_{\FF}(\Ker(\phi))_j= \dim_{\FF}\big(\FF[D,\partial D]/\Sigma\big(\Theta; \FF[D,\partial D]\big)\big)_j$ for all $j$.
\end{enumerate}

For brevity, let $R=\FF[\Delta]$ and let $I=\FF[D,\partial D]$. {Since $(D,\partial D)=(\Delta,\Gamma)$ as relative simplicial complexes, it follows that} $I$ is an ideal of $R$, and $R/I$ is $\FF[\Gamma]$. Thus our surjection $\phi$ is the projection $\phi:  R/\Sigma(\Theta;R) \to R/(I+\Theta R)$. Hence 
\[ \Ker(\phi)=
(I+\Theta R)/\Sigma(\Theta;R) = (I+ \Sigma(\Theta;R))/\Sigma(\Theta;R) = I/(I \cap \Sigma(\Theta;R)),
\]
which together with an observation that $I \cap \Sigma(\Theta;R)$ contains $\Sigma(\Theta;I)$ yields assertion (i). To see  that $I \cap \Sigma(\Theta;R)\supseteq \Sigma(\Theta;I)$, note that since $I$ is a subset of $R$, $\Theta I$ is contained in $\Theta R$, and $(\theta_1,...,\hat \theta_i,\dots,\theta_d)I :_I \theta_i$  is contained in $(\theta_1,...,\hat \theta_i,\dots,\theta_d)R :_R \theta_i$ for all $i\in[d]$; therefore $\Sigma(\Theta;R) \supseteq \Sigma(\Theta;I)$.
 
As for assertion (ii), the dimension of the kernel of $\phi$ can be computed as follows: by \eqref{surjection} and by Theorem \ref{3.3}, 
\begin{eqnarray} %\nonumber
\dim_{\FF}(\Ker(\phi))_j = h''_j(\Delta)-h_j(\Gamma)
\label{Kernel}
\end{eqnarray}
Now, since each face of $\Delta$ is either a face of $(D,\partial D)$ or a face of $\Gamma$ (but not of both), $f_i(\Delta)=f_i(D,\partial D)+ f_i(\Gamma)$ for all $i\geq -1$, and so 
\begin{equation} \label{h's}
h_j(D,\partial D)=h_j(\Delta)-h_j(\Gamma) \quad \mbox{for all $j\geq 0$}.
\end{equation}
%Substituting \eqref{h's} and \eqref{Betti} in \eqref{Kernel} yields
 Equations \eqref{Betti}, \eqref{Kernel} and \eqref{h's} yield
\[
\dim_{\FF}(\Ker(\phi))_j =h''_j(D, \partial D)= \dim_{\FF}\big(\FF[D,\partial D]/\Sigma\big(\Theta; \FF[D,\partial D]\big)\big)_j \quad \mbox{for all $j\geq 0$},
\]
where the last equality is another application of Theorem \ref{3.3}. The assertion follows.\end{proof}

Another ingredient needed for both proofs is the following  immediate consequence of the snake lemma.

\begin{lemma} \label{snake}
Let $0\to L \to N \to M\to 0$ be an exact sequence of graded $S$-modules, let $\omega\in S$ be a linear form, and let $k$ be a fixed integer. Assume also that the map  $\times \omega: M_k\to M_{k+1}$ is bijective. Then the map $\times \omega: N_k\to N_{k+1}$ is surjective if and only if the map $\times \omega: L_k\to L_{k+1}$ is surjective.
\end{lemma}

We are now ready to prove both of the theorems.

\begin{proof}[Proof of Theorem \ref{flips}]
 We are given that $\Delta'=\big(\Delta \setminus (\overline{A}\ast\partial\overline{B})\big) \cup \big(\partial \overline{A}\ast \overline{B}\big)$, where $|B|=p$.
Let $D$ be the simplicial complex obtained from $\Delta$ by removing the interior faces of the ball $\Gamma_1=\overline{A}\ast\partial\overline{B}$; equivalently, $D$ is obtained from $\Delta'$ by removing the interior faces of  $\Gamma_2=\partial \overline{A}\ast \overline{B}$.
Since $\Gamma_1$ and $\Gamma_2$ are full-dimensional subcomplexes of $\Delta$ and $\Delta'$, and $\Theta$ is an l.s.o.p.\ for both $\FF[\Delta]$ and $\FF[\Delta']$, it is also an l.s.o.p.~for both $\FF[\Gamma_1]$ and $\FF[\Gamma_2]$.
Hence
$$\dim_\FF (\FF[\Gamma_1]/\Theta \FF[\Gamma_1])_j=h_j(\overline{A}\ast \partial\overline{B})=h_j(\partial\overline{B})=\left\{\begin{array}{ll} 1, & \mbox{ if $j \leq p-1,$}\\ 0, & \mbox{ if $j > p-1,$} \end{array}\right.$$
which implies that $\FF[\Gamma_1]/\Theta \FF[\Gamma_1] \cong \FF[x]/(x^p)$.
Similarly, $\FF[\Gamma_2]/\Theta \FF[\Gamma_2] \cong \FF[x]/(x^{d-p+1})$.
These isomorphisms together with the assumption $p \not \in \{ \frac d 2, \frac {d+1} 2, \frac {d+2} 2\}$
yield that,
for a generic choice of a linear form $w$,
the multiplication map
$$\times w : (\FF[\Gamma_i]/\Theta \FF[\Gamma_i])_{\lfloor \frac d 2 \rfloor}
\to (\FF[\Gamma_i]/\Theta \FF[\Gamma_i])_{\lfloor \frac d 2 \rfloor +1}$$
is bijective for $i \in \{1,2\}$.

Applying Lemma \ref{exact} to $\Delta$ and $\Gamma_1$, we conclude from Lemma \ref{snake} that
$\FF[\Delta]/\Sigma(\Theta;\FF[\Delta])$ has the WLP if and only if, for a generic linear form $w$, the multiplication
$$\times w : \big(\FF[D,\partial D]/\Sigma(\Theta;\FF[D,\partial D]) \big)_{\lfloor \frac d 2 \rfloor} 
\to \big(\FF[D,\partial D]/\Sigma(\Theta;\FF[D,\partial D])\big)_{\lfloor \frac d 2 \rfloor+1}$$
is surjective.
On the other hand, by Lemma \ref{exact} applied to $\Delta'$ and $\Gamma_2$, the latter condition is equivalent to the WLP of $\FF[\Delta']/\Sigma(\Theta;\FF[\Delta'])$. 
\end{proof}

\begin{proof}[Proof of Theorem \ref{stellar-subdiv}]
Let $D$ be the homology manifold obtained from $\Delta$ by removing the interior faces of the homology ball $\Gamma_1=\st_\Delta(\sigma)$;
equivalently, $D$ is obtained from $\sd_\sigma(\Delta)$ by removing the interior faces of the homology ball $\Gamma_2=\overline{a}\ast\partial\overline{\sigma}\ast\lk_\Delta(\sigma)$.
As the proof of Theorem \ref{flips} shows, to verify Theorem \ref{stellar-subdiv}, it suffices to check that, for a generic choice of linear forms $\Theta=\theta_1,\dots,\theta_d$ and another generic linear form $w$,
the multiplication
\begin{align}
\label{bijection}
\times w : \big(\FF[\Gamma_i]/\Theta \FF[\Gamma_i]\big)_{\lfloor \frac d 2 \rfloor}
\to \big(\FF[\Gamma_i]/\Theta \FF[\Gamma_i]\big)_{\lfloor \frac d 2 \rfloor +1}
\end{align}
is bijective for each $i \in \{1,2\}$.
%,where $\Theta=\theta_1,\dots,\theta_d$.

In the case of $i=1$, the desired bijection is an immediate consequence of the fact that
$$\dim_\FF (\FF[\Gamma_1]/\Theta \FF[\Gamma_1])_j=h_j(\st_\Delta(\sigma))=h_j(\lk_\Delta(\sigma))=0 \quad
\mbox{for }  j \geq d - |\sigma|+1$$
and the assumption $|\sigma| > \frac d 2 +1$.
We now prove that the map in \eqref{bijection} is also bijective for $i=2$.
To this end, observe that the homology $(d-2)$-sphere $\partial \overline \sigma * \lk_\Delta(\sigma)$ is the boundary of the homology $(d-1)$-ball $\overline \sigma * \lk_\Delta(\sigma)$, and hence
$\partial \overline \sigma * \lk_\Delta(\sigma)$ is ($d-|\sigma|)$-stacked. (A homology $(d-2)$-sphere  is called $(d-k)$-stacked if it is the boundary of a homology ball that has no interior faces of size $\leq k-1$.)
It then follows from \cite[Corollary 6.3]{Sw} that $\partial \overline \sigma * \lk_\Delta(\sigma)$ has the WLP. This, in turn, implies that the map in \eqref{bijection} is surjective as $\FF[\Gamma_2]$ is a polynomial ring over $\FF[\partial \overline \sigma * \lk_\Delta(\sigma)]$.
Also, since $|\sigma|> \frac{d}{2}+1$, we infer from \cite[Theorem 2]{MW} and the $(d-|\sigma|)$-stackedness of $\partial \overline \sigma * \lk_\Delta(\sigma)$ that
$h_{\lfloor \frac d 2 \rfloor}(\partial \overline \sigma * \lk_\Delta(\sigma))
=h_{\lfloor \frac d 2 \rfloor+1}(\partial \overline \sigma * \lk_\Delta(\sigma))$. As $\Gamma_2$ and $\partial \overline \sigma * \lk_\Delta(\sigma)$ have the same $h$-vector, we conclude that
$h_{\lfloor \frac d 2 \rfloor}(\Gamma_2)=h_{\lfloor \frac d 2 \rfloor+1}(\Gamma_2)$,
and therefore that the map in \eqref{bijection} is bijective.
\end{proof}

\smallskip
Let $\Delta$ be a triangulation of a closed surface and assume that one of the vertices of $\Delta$ is connected to all other vertices of $\Delta$. Then by  \cite[Theorem 1.6]{NS2}, $\Delta$ has the WLP. However, the problem of whether {\em every} triangulation of a closed surface other than the sphere possesses the WLP is at present wide open.       

\begin{remark} Note that %if $\FF$ is a field of characteristic zero and 
if $\Delta$ is an arbitrary {\em odd-dimensional} 
Buchsbaum complex (e.g., a homology manifold), then there exists a simplicial complex $\Gamma$ that is PL homeomorphic to $\Delta$ and has the WLP in characteristic $0$. Simply take $\Gamma$ to be the barycentric subdivision of $\Delta$: the WLP of $\Gamma$ is guaranteed by Theorem~\ref{baryc-subdiv}.
\end{remark}

%It was observed by Ed Swartz (personal communication) that if $\Delta$ is an arbitrary odd-dimensional PL manifold, then we can always find a simplicial complex $\Gamma$ such that (i) $\Gamma$ is PL homeomorphic to $\Delta$ and (ii) $\Gamma$ has the WLP. Indeed, denote by $\bsd(\Delta)$ the barycentric subdivision of $\Delta$. According to \cite[Corollary 2]{AI}, there exists a number $k$ (that depends on $\Delta$) with the property that all vertex links of the complex $\Gamma:=\bsd^k(\Delta)$ --- the $k$th barycentric subdicision of $\Delta$ ---  are polytopal spheres. Thus, by a celebrated result of Stanley \cite{St80}, all vertex links of $\Gamma$ have the WLP. As $\Gamma$ is odd-dimensional, it then follows from \cite[Theorem 4.26]{Sw} that $\Gamma$ itself has the WLP.

\section*{Appendix. Proof of Theorem \ref{2.3}}
\renewcommand{\thesection}{\Alph{section}}
\setcounter{section}{1}
\setcounter{theorem}{0}

The goal of this Appendix is to verify Theorem \ref{2.3}.
Our proof is based on the proof of \cite[Theorem 2.2]{NS}, and so some details are omitted. Let $M$ be a finitely generated Buchsbaum graded $S$-module of Krull dimension $d$,
and let $\Theta=\theta_1,\dots,\theta_d$ be an h.s.o.p.\ for $M$ with $\deg \theta_i=\ddd_i$.
We write $\ddd_C=\sum_{i \in C} \ddd_i$ for $C \subseteq [d]$.
We need the following result (see \cite[Lemma II.4.14$'$]{SV}).

\begin{lemma}
\label{A.1} There is an isomorphism
$$H_\mideal^i(M/((\theta_1,\dots,\theta_j)M))\cong \bigoplus_{C \subseteq [j]} H_\mideal^{|C|+i}(M)(-\ddd_C) \quad
\mbox{for all $i,j$ with $i+j <d$}.$$
\end{lemma}

We use the following notation: if $C \subseteq [d]$, then $M\langle C\rangle$ is defined by
$$M\langle C\rangle = M/\big((\theta_i: i \not \in C)M\big).$$
By \cite[Proposition I.2.1]{SV}, if 
 $C \subsetneq [d]$ and $s \in [d]\setminus C$,
then $H_\mideal^0(M\langle C \cup \{s\}\rangle ) = 0:_{M\langle C \cup \{s\}\rangle} \theta_s$. This leads to the following short exact sequence
\begin{align}
\label{short}
0 \longrightarrow M\langle C\cup\{s\}\rangle/H_\mideal^0( M\langle C\cup\{s\}\rangle) (-\ddd_s) \stackrel{\times \theta_s\ }  \longrightarrow M\langle C \cup \{s\}\rangle \stackrel{\pi_s} \longrightarrow M\langle C\rangle \longrightarrow 0,
\end{align}
where $\pi_s$ is a natural projection.
This short exact sequence gives rise to exact sequences
\begin{align}
\label{A-1}
0 \longrightarrow H_\mideal^k (M\langle C\cup\{s\}\rangle) \stackrel {\pi^*_s} \longrightarrow H_\mideal^k(M\langle C\rangle ) \stackrel {\varphi_s^*} \longrightarrow H_\mideal^{k+1} (M\langle C\cup\{s\}\rangle) (-\ddd_s) \longrightarrow 0
\end{align}
for $k< |C|$, and
\begin{align}
\label{A-2}
0 \longrightarrow H_\mideal^{|C|} (M\langle C\cup\{s\}\rangle) \stackrel {\pi^*_s} \longrightarrow H_\mideal^{|C|} (M\langle C\rangle ) \stackrel {\varphi_s^*} \longrightarrow H_\mideal^{|C|+1} (M\langle C\cup\{s\}\rangle) (-\ddd_s),
\end{align}
where we denote by $\varphi_s^*$ the connecting homomorphism.

Similarly, if $C \subsetneq [d]$, $|C| \leq d-2$, and  $s,t \in [d]\setminus C$,
we obtain the following commutative diagram
{\small
\begin{align*}
\begin{array}{cccccccccc}
0  &\to& M\langle C\cup \{s\}\rangle /H_\mideal^0(M\langle C\cup \{s\}\rangle){(-\ddd_s)}  &\stackrel {\times \theta_s} \to& M\langle C\cup \{s\}\rangle &\stackrel{\pi_s} \to & M\langle C\rangle & \to 0\\
& & \uparrow\pi_t & & \uparrow \pi_t & & \uparrow  \pi _t\\
0  &\to& M\langle C\cup\{s,t\}\rangle /H_\mideal^0(M\langle C\cup\{s,t\}\rangle){(-\ddd_s)}  &\stackrel {\times \theta_s} \to& M\langle C\cup\{s,t\}\rangle &\stackrel{\pi_s}  \to & M \langle C \cup \{t\} \rangle & \to 0.
\end{array}
\end{align*}
}This diagram, in turn, induces the commutative diagram
\begin{eqnarray}
\label{A-3}
\begin{array}{ccccccc}
 H_\mideal^i(M \langle C\rangle) & \stackrel {\varphi^*_s} \longrightarrow & H_\mideal^{i+1} (M\langle C \cup \{s\}\rangle )(-\ddd_s)\\
\uparrow \pi_t^* & & \uparrow \pi_t^*\\
H_\mideal^i(M\langle C\cup \{t\}\rangle) & \stackrel {\varphi^*_s} \longrightarrow & H_\mideal^{i+1}(M\langle C \cup \{s,t\}\rangle) (-\ddd_s).
\end{array}
\end{eqnarray}

Now, for $k=2,3,\dots,d$, define the maps $\phi_k$ and $\psi_k$ as compositions
$$\phi_k: H_\mideal^0(M\langle \emptyset\rangle) \stackrel {\varphi_1^*} \to H_\mideal^1 \big(M\big\langle [1]\big\rangle\big) (-\ddd_{[1]}) \stackrel {\varphi_2^*} \to \cdots
\stackrel {\varphi_{k-1}^*} \to H_\mideal^{k-1} \big(M\big\langle [k-1]\big\rangle\big) (-\ddd_{[k-1]})
$$
and
{\small
$$\psi_k: H_\mideal^0(M\langle\{k\}\rangle) \stackrel {\varphi_1^*} \to H_\mideal^1 \big(M\big\langle [1] \cup \{k\} \big\rangle\big) (-\ddd_{[1]}) \stackrel {\varphi_2^*} \to \cdots
\stackrel {\varphi_{k-1}^*} \to H_\mideal^{k-1} \big(M\big\langle[k]\big\rangle\big) (-\ddd_{[k-1]}).
$$
}
Then the commutativity of \eqref{A-3} implies the commutativity of
\begin{eqnarray}
\label{A-4}
\begin{array}{ccccccc}
H_\mideal^0 (M\langle\emptyset\rangle) & \stackrel {\phi_k} \longrightarrow & H_\mideal^{k-1}\left(M\big\langle[k-1]\big\rangle\right) (-\ddd_{[k-1]})\medskip\\
\uparrow \pi_k^* & & \uparrow \pi_k^*\medskip\\
H_\mideal^0(M\langle\{k\}\rangle)  & \stackrel {\psi_k} \longrightarrow & 
H_\mideal^{k-1}\left(M\big\langle [k]\big\rangle\right)(-\ddd_{[k-1]}).
\end{array}
\end{eqnarray}

\bigskip

To prove part (ii) of Theorem \ref{2.3}, we consider the following diagram
\begin{eqnarray*}
\begin{array}{cccclcc}
& &H_\mideal^0(M\langle\{1\}\rangle) & \stackrel{\pi^*_1} \longrightarrow &H_\mideal^0(M\langle\emptyset\rangle)\smallskip\\
& & & & \hspace{30pt}\downarrow \varphi_1^*\smallskip\\
& H_\mideal^0\big(M\langle\{2\}\rangle\big) \stackrel {\psi_2} \longrightarrow &H_\mideal^1 \big(M\big\langle[2]\big\rangle\big) (-\ddd_{[1]}) & \stackrel{\pi^*_2}  \longrightarrow &H_\mideal^1\big(M\big\langle[1]\big\rangle\big)(-\ddd_{[1]})\smallskip\\
& & & & \hspace{30pt}\downarrow \varphi_2^*\smallskip \\
& & & & \hspace{32pt} \vdots \smallskip \\
& & & & \hspace{30pt} \downarrow \varphi_{d-1}^*\smallskip \\
& H_\mideal^0\big(M\langle\{d\}\rangle\big) \stackrel {\psi_d} \longrightarrow &H_\mideal^{d-1} \big(M\big\langle[d]\big\rangle\big)(-\ddd_{[d-1]})  & \stackrel{\pi^*_d} \longrightarrow &H_\mideal^{{d-1}}\big(M\big\langle[d-1]\big\rangle\big) (-\ddd_{[d-1]})\smallskip\\
& & & & \hspace{30pt}\downarrow \varphi_{d}^* \medskip \\
& & &  &H_\mideal^d\left(M\big\langle[d] \big\rangle\right) (-\ddd_{[d]}).
\end{array}
\end{eqnarray*}
The surjectivity of $\varphi_s^*$ in \eqref{A-1}
implies that $\psi_k$ is surjective. Hence in each horizontal line of the diagram,
\begin{align}
\label{A-5}
\Image (\pi_k^* \circ \psi_k) = \Image (\pi_k^*).
\end{align}
 Also, since the sequence in \eqref{A-2} is exact, it follows that in the diagram,
\begin{align}
\label{A-6}
\Ker (\varphi_{k}^*) = \Image (\pi_k^*).
\end{align}

Define $\phi_{d+1}= \varphi_d^* \circ \cdots \circ\varphi_1^*$ to be the composition of the vertical maps in the diagram. By \eqref{A-2}, each $\pi_k^*$ (in the diagram) is injective, and we conclude that
\begin{align}
\label{A-7}
\Ker (\phi_{d+1}) \cong \bigoplus_{k=1}^d \Image (\pi_k^*) \cong \bigoplus_{k=1}^d H_\mideal^{k-1}\left( M\big\langle [k] \big\rangle \right) (-\ddd_{[k-1]})
\end{align}
as $\field$-vector spaces.
In addition, using \eqref{A-5} and the commutativity of \eqref{A-4}, we obtain that $\Ker( \phi_{d+1})$ is the sum of the images of
$$\pi_k^* : H_\mideal^0(M\langle\{k\}\rangle) \to H_\mideal^0(M\langle\emptyset\rangle)=M/\Theta M.$$
Finally, since 
$$H_\mideal^0(M\langle\{k\}\rangle) = \big( (\theta_1,\dots,\hat \theta_k,\dots,\theta_d)M:_M \theta_k \big) / (\theta_1,\dots,\hat \theta_k,\dots,\theta_d)M$$
and since $\pi_k^*$ is a natural projection, it follows that
$$\Ker (\phi_{d+1}) = \Sigma(\Theta;M)/\Theta M.$$
This proves part (ii) of the statement since $\phi_{d+1}$ is a map from $M/\Theta M$ to $H_\mideal^d(M) (-\ddd_{[d]})$.

It remains to verify part (i). Recall that $\Ker (\phi_{d+1})$ is the sum of images of $H_\mideal^0(M\langle\{k\}\rangle)$ and that by \cite[Proposition I.2.1]{SV}, $\mideal \cdot H_\mideal^0(M\langle\{k\}\rangle)=0$.
Hence the modules in \eqref{A-7} are in fact isomorphic as $S$-modules (since they are direct sums of copies of $\field$). Thus we infer from
\eqref{A-7} and Lemma \ref{A.1} that
\begin{align*}
\Sigma(\Theta;M) /\Theta M 
= \Ker(\phi_{d+1}) 
&\cong \bigoplus_{k=1}^d H_\mideal^{k-1} \left(M\big\langle[k]\big\rangle\right)(-\ddd_{[k-1]})\\
&\cong \bigoplus_{k=1}^d \left[ \bigoplus_{C \subseteq [d]\setminus [k]} H_\mideal^{|C|+k-1} (M) (-\ddd_{C \cup [k-1]}) \right]\\
&\cong \bigoplus_{C \subsetneq [d]} H_\mideal^{|C|} (M) (-\ddd_C),
\end{align*}
as desired.
\hfill $\square$

\begin{remark2}
As the proof in this Appendix is based on the proof of \cite[Theorem 2.2]{NS}, it is worth pointing out that there is a minor mistake in \cite{NS}.
Indeed, the short ``exact" sequence that appears three lines after the statement of Theorem 2.4 in \cite{NS} is not necessarily exact.
However, this mistake can be easily corrected by replacing this sequence with the short exact sequence of equation \eqref{short}.
\end{remark2}
\bigskip

\noindent\textbf{Acknowledgments}:
We thank Ed Swartz for raising a question of whether there is an algebraic version of Theorem \ref{MuraiNovik} and for bringing the result of  B\"ohm and Papadakis to our attention.
We are also grateful to Shiro Goto for helpful conversations.
%The first author was partially supported by JSPS KAKENHI 25400043.

{\small
\bibliography{BBMduality-biblio}
\bibliographystyle{plain}
}
\end{document}